\documentclass[11pt, a4paper]{amsart}

\usepackage{amsfonts,amsmath,amssymb, amscd}
\usepackage{txfonts, mathabx}

\newtheorem{theorem}{Theorem}[section]
\newtheorem{lemma}[theorem]{Lemma}
\newtheorem{definition}[theorem]{Definition}
\newtheorem{prop}[theorem]{Proposition}
\newtheorem{coro}[theorem]{Corollary}

\theoremstyle{definition}
\newtheorem{rem}[theorem]{Remark}

\newcommand\NN{\mathbb{N}}
\newcommand\ZZ{\mathbb Z}

\newcommand\CC{\mathbb{C}}
\newcommand\FF{\mathbb{F}}

\renewcommand\AA{\mathbb{A}}

\newcommand\ii{\textbf{\textit{i}}}

\DeclareMathOperator{\Hilb}{Hilb}

\numberwithin{equation}{section}

\title[The zeta function of a punctual Hilbert scheme of the torus]
{Complete determination of\\ the zeta function of the Hilbert scheme of $n$~points\\ 
on a two-dimensional torus}

\author{Christian Kassel}
\address{Christian Kassel: 
Universit\'e de Strasbourg, CNRS, IRMA UMR 7501, F--67000 Strasbourg, France}
\email{kassel@math.unistra.fr}
\urladdr{www-irma.u-strasbg.fr/\raise-2pt\hbox{\~{}}kassel/}

\author{Christophe Reutenauer}
\address{Christophe Reutenauer:
Math\'ematiques, Universit\'e du Qu\'ebec \`a Montr\'eal,
Montr\'eal, CP 8888, succ.\ Centre Ville, Canada H3C 3P8}
\email{reutenauer.christophe@uqam.ca}
\urladdr{www.lacim.uqam.ca/\raise-2pt\hbox{\~{}}christo/}



\begin{document}

\begin{abstract}
We compute the coefficients of the polynomials~$C_n(q)$ defined by the equation
\begin{equation*}
1 + \sum_{n\geq 1} \, \frac{C_n(q)}{q^n}  \, t^n
= \prod_{i\geq 1}\, \frac{(1-t^i)^2}{1-(q+q^{-1})t^i + t^{2i}} \, .
\end{equation*}
As an application we obtain an explicit formula for the zeta function of the Hilbert scheme of
$n$~points on a two-dimensional torus and show that this zeta function satisfies a remarkable functional equation.
The polynomials~$C_n(q)$ are divisible by~$(q-1)^2$.
We also compute the coefficients of the polynomials $P_n(q) = C_n(q)/(q-1)^2$:
each coefficient counts the divisors of~$n$ in a certain interval; it is thus a non-negative integer. 
Finally we give arithmetical interpretations for the values of~$C_n(q)$ and of~$P_n(q)$ at $q = -1$
and at roots of unity of order $3$, $4$,~$6$.
\end{abstract}

\maketitle

\section{Introduction}

Let $\FF_q$ be a finite field of cardinality~$q$ and $\FF_q[x,y,x^{-1}, y^{-1}]$ be the algebra of
Laurent polynomials in two variables with coefficients in~$\FF_q$.
In\,\cite[Th.\,1.1]{KRidealcount} we gave the  following formula for the number~$C_n(q)$ of ideals of codimension~$n$
of~$\FF_q[x,y,x^{-1}, y^{-1}]$, where $n\geq 1$:
\begin{equation}\label{formula-Cnq}
C_n(q) = \sum_{\lambda \,\vdash n} \,
(q-1)^{2 v(\lambda)} \, q^{n - \ell(\lambda)} \,\prod_{i= 1, \ldots, t \atop d_i \geq 1} \, \frac{q^{2d_i} - 1}{q^2 - 1} \, ,
\end{equation}
where the sum runs over all partitions~$\lambda$ of~$n$
(the notation $\ell(\lambda)$, $\nu(\lambda)$, $d_i$ are defined in \emph{loc. cit.}).
This formula has the following immediate consequences.

\begin{itemize}
\item[(i)]
$C_n(q)$ is a monic polynomial in the variable~$q$ with integer coefficients and of degree~$2n$.

\item[(ii)]
The polynomial $C_n(q)$ is divisible by $(q-1)^2$.

\item[(iii)]
If we set 
\begin{equation*}
P_n(q) = \frac{C_n(q)}{(q-1)^2} \, ,
\end{equation*}
then $P_n(q)$ is a monic polynomial with integer coefficients and of degree~$2n-2$.

\item[(iv)]
The value at~$q=1$ of $P_n(q)$ is given by
\begin{equation*}\label{eq-Pn1}
P_n(1) = \sigma(n) = \sum_{d | n\, ;\, d\geq1} \, d .
\end{equation*}
\end{itemize}

The aim of the present article is to compute the coefficients of the polynomials~$C_n(q)$ and~$P_n(q)$.
To this end we use another consequence of\,\eqref{formula-Cnq}, namely the following infinite product
expression we obtained in \emph{loc. cit.} for the generating function of the polynomials~$C_n(q)$
(see\,\cite[Cor.\,1.3]{KRidealcount}):
\begin{equation}\label{gf-Cnq}
1 + \sum_{n\geq 1} \, \frac{C_n(q)}{q^n}  \, t^n
= \prod_{i\geq 1}\, \frac{(1-t^i)^2}{1-(q+q^{-1})t^i + t^{2i}} \, .
\end{equation}

This infinite product is clearly invariant under the transformation $q \leftrightarrow q^{-1}$.
Therefore, 
\begin{equation*}
\frac{C_n(q)}{q^n} = \frac{C_n(q^{-1})}{q^{-n}} \, ,
\end{equation*}
which implies that the polynomials~$C_n(q)$ are palindromic.
We can thus express them as follows:
\begin{equation}\label{def-cni}
C_n(q) = c_{n,0} \, q^n + \sum_{i=1}^n \, c_{n,i} \, \left( q^{n+i} + q^{n-i} \right),
\end{equation}
where the coefficients $c_{n,i}$ ($0 \leq i \leq n$) are integers.

Our first main theorem is the following.

\begin{theorem}\label{th-cni}
Let $c_{n,i}$ be the coefficients of the Laurent polynomial~$C_n(q)$ as defined by\,\eqref{def-cni}.

(a) For the central coefficients we have
\begin{equation*}
c_{n,0} =
\left\{
\begin{array}{cl}
2\, (-1)^k  & \text{if}\;\, n = k(k+1)/2 \;\; \text{for some integer}\; k \geq 1, 
\\
\noalign{\smallskip}
0 & \text{otherwise.}
\end{array}
\right.
\end{equation*}

(b) For the non-central coefficients ($i\geq 1$) we have
\begin{equation*}
c_{n,i} =
\left\{
\begin{array}{cl}
(-1)^k & \text{if}\;\, n = k(k+2i +1)/2 \;\; \text{for some integer}\; k  \geq 1, \\
\noalign{\smallskip}
(-1)^{k-1} & \text{if}\;\, n = k(k+2i -1)/2 \;\; \text{for some integer}\; k \geq 1, \\
\noalign{\smallskip}
0 & \text{otherwise.} 
\end{array}
\right.
\end{equation*}
\end{theorem}

Note that the coefficients of~$C_n(q)$ take only the values $0$, $\pm 1$, and~$\pm 2$,
and that in Item\,(c) the first two conditions are mutually exclusive.
A list of polynomials $C_n(q)$ ($1 \leq n \leq 12$) appears in Table\,\ref{tableC}.

In the next section we will apply Theorem\,\ref{th-cni} to give an explicit formula for the local zeta function of
the Hilbert scheme of $n$~points on a two-dimensional torus (see Theorem\,\ref{th-mainZ}).

Since the polynomial $P_n(q)$ is the quotient of two palindromic polynomials,
$P_n(q)$ is palindromic as well. We can thus expand it as follows:
\begin{equation*}
P_n(q) = a_{n,0}\, q^{n-1}  + \sum_{i=1}^{n-1} \, a_{n,i} \, \left( q^{n+i-1} + q^{n-i-1} \right) , 
\end{equation*}
where $a_{n,i}$ ($0 \leq i \leq n-1$) are integers.

We now state our second main theorem.

\begin{theorem}\label{th-ani} 
For $n\geq 1$ and $0 \leq i \leq n-1$, 
the integer~$a_{n,i}$ is equal to the number of divisors~$d$ of~$n$ satisfying the inequalities
\begin{equation*}
\frac{i+ \sqrt{2n+i^2}}{2} < d \leq i+ \sqrt{2n+i^2} \, .
\end{equation*}
\end{theorem}

The most striking consequence of this theorem is that the coefficients of~$P_n(q)$ are \emph{non-negative} integers.
Table\,\ref{tableP} lists the polynomials~$P_n(q)$ for $1 \leq n \leq 12$
(this table also displays their values at~$-1$, at a primitive third root of unity~$j$ and at $\ii = \sqrt{-1}$).

\begin{rem}
It follows from Theorem\,\ref{th-ani} that the central coefficient~$a_{n,0}$ of~$P_n(q)$ is equal
to the number of ``middle divisors'' of~$n$, 
i.e. of the divisors~$d$ in the half-open interval $(\!\sqrt{n/2},\sqrt{2n}\,]$.
The table\footnote{This table is due to R.~Zumkeller} in~\cite[A067742]{OEIS} 
listing the central coefficients~$a_{n,0}$ for $n \leq 10000$ 
suggests that the sequence $(a_{n,0})_n$ grows very slowly (see also~\cite[A128605]{OEIS});
Vatne recently proved that this sequence is unbounded (see\,\cite{Va}).
\end{rem}

\begin{rem}
Since the degree of~$P_n(q)$ is $2n-2$, its coefficients are non-negative integers and their sum~$P_n(1)$ 
is equal to the sum~$\sigma(n)$ of divisors of~$n$, 
necessarily at least one of the coefficients of~$P_n(q)$ must be $\geq 2$ when $\sigma(n) \geq 2n$, i.e.
when $n$ is a perfect number (such as $6$ or $28$) or an abundant number (such as $12$, $18$, $20$, $24$ or~$30$).
\end{rem}

Our third main result is concerned with the values of~$C_n(q)$ and~$P_n(q)$
at certain roots of unity. 
Indeed, if $\omega$ is a root of unity of order $d = 2$, $3$, $4$, or $6$, then 
$\omega + \omega^{-1} \in \ZZ$, which together with\,\eqref{gf-Cnq}
shows that $C_n(\omega)/\omega^n$ is an integer. 
It is easy to check that when we set $q = \omega$ in\,\eqref{gf-Cnq} the infinite product
is expressible in terms of Dedekind's eta function $\eta(z) = t^{1/24} \prod_{n \geq 1}\, (1-t^n)$
where $t = e^{2\pi iz}$.
More precisely,
\begin{equation}\label{eq-etaproduct}
1 + \sum_{n\geq 1} \, \frac{C_n(\omega)}{\omega^n}  \, t^n = 
\left\{
\begin{array}{ccl}
\displaystyle{\frac{\eta(z)^4}{\eta(2z)^2} } && \text{if} \; d = 2 , \\
\noalign{\medskip}
\displaystyle{\frac{\eta(z)^3}{\eta(3z)} }&& \text{if} \; d = 3 , \\
\noalign{\medskip}
\displaystyle{\frac{\eta(z)^2 \, \eta(2z)}{\eta(4z)} } && \text{if} \; d = 4 , \\
\noalign{\medskip}
\displaystyle{\frac{\eta(z) \, \eta(2z) \, \eta(3z)}{\eta(6z) }} && \text{if} \; d = 6 .
\end{array}
\right.
\end{equation}
The eta-products appearing in the previous equality are \emph{modular forms} of weight~one and
of level~$d$ (see\,\cite{Ko}).
We will determine the four sequences $(C_n(\omega))_{n\geq 1}$ explicitely (see Theorem\,\ref{th-omega}).
In view of their relationship with modular forms, it is not surprising that these sequences are
related to well-known arithmetical sequences.

In order to state Theorem\,\ref{th-omega} we introduce some notation.

\begin{itemize}
\item[(i)] 
Let~$r(n)$ be the number of representations of~$n$ as a sum of two squares:
\begin{equation}\label{def-r1n}
r(n) = \left\{ (x,y) \in \ZZ^2 \, |\, x^2 + y^2 =n  \right \}.
\end{equation}
The sequence~$r(n)$ is Sequence~A004018 of\,\cite{OEIS}. 
Its generating function is the theta series of the square lattice.
Note that $r(n)$ is divisible by~$4$ for $n\geq 1$ due to the symmetries of the square lattice.

\item[(ii)] 
Similarly, we denote by~$r'(n)$ the number of representations of~$n$ as a sum of a square and twice another square:
\begin{equation}\label{def-r2n}
r'(n) = \left\{ (x,y) \in \ZZ^2 \, |\, x^2 + 2y^2 =n  \right \}.
\end{equation}
The sequence $r'(n)$ is Sequence~A033715 of\,\cite{OEIS}.

\item[(iii)]  
Let $\lambda(n)$ be the multiplicative function of~$n$ defined by on all prime powers by
\begin{equation}\label{def-lambda}
\lambda(p^e) = 
\begin{cases}
\hskip 15pt -2 & \text{if} \; p=3 \; \text{and}\; e\geq 1,  \\
\noalign{\medskip}
\hskip 15pt  e+1\ & \text{if} \; p \equiv 1 \pmod{6}, \\
\noalign{\medskip}
(1+(-1)^e)/2 & \text{if} \; p \equiv 2,5 \pmod{6}.
\end{cases}
\end{equation}
Recall that ``multiplicative'' means that $\lambda(mn) = \lambda(m)\lambda(n)$ whenever $m$ and $n$
are coprime. The function $\lambda(n)$ can also be expressed in terms of the excess function
\begin{equation}\label{def-excess}
E_1(n;3) 
= \sum_{d|n\atop d \equiv 1\; \text{mod} \, 3} \, 1 -  \sum_{d|n\atop d \equiv 2\; \text{mod} \, 3} \, 1.
\end{equation}
Indeed, $\lambda(n) =  E_1(n;3) - 3 E_1(n/3;3)$, where we agree that $E_1(n/3;3) = 0$ 
when $n$ is not divisible by~$3$.
\end{itemize}

We can now state our results concerning the values of~$C_n(q)$ and~$P_n(q)$ 
at the roots of unity of order $d = 2$, $3$, $4$,~$6$.
We express these values in terms of the notation we have just introduced.

\begin{theorem}\label{th-omega} 
Let $n$ be an integer $\geq 1$.

(a) We have
\begin{equation*}
C_n(-1) = r(n) 
\quad\text{and}\quad
P_n(-1) = \frac{r(n)}{4}\, . 
\end{equation*}

(b) Let $j= e^{2\pi i/3}$, a primitive third root of unity. Then 
\begin{equation*}
C_n(j) = -3 \lambda(n)  j^n
\quad\text{and}\quad
P_n(j) = \lambda(n)  j^{n-1}. 
\end{equation*}

(c) Let $\ii = \sqrt{-1}$ be a square root of~$-1$. 
Then\footnote{If $\alpha$ is a real number, then $\lfloor \alpha \rfloor$ stands for the largest integer $\leq \alpha$.}
\begin{equation*}
C_n(\ii) = (-1)^{\lfloor (n+1)/2\rfloor}  \, r'(n) \, \ii^n
\quad\text{and}\quad
P_n(\ii) = (-1)^{\lfloor (n-1)/2\rfloor}  \, \frac{r'(n)}{2} \, \ii^{n-1}.
\end{equation*}

(d) Let $-j = e^{\pi i/3}$, a primitive sixth root of unity. We have
\begin{equation*}
C_n(-j) = 
\begin{cases}
\hskip 12pt  r(n)  & \text{if} \; n \equiv 0,  \\
\noalign{\medskip}
\hskip 12pt   \displaystyle{\frac{r(n)}{4}}  j & \text{if} \; n \equiv 1, \quad\pmod{3}\\
\noalign{\medskip}
\hskip 5pt -\displaystyle{\frac{r(n)}{2}}  j^2 & \text{if} \; n \equiv 2,
\end{cases}
\end{equation*}
and $P_n(-j) = C_n(-j)/j = C_n(-j)\, j^2$.
\end{theorem}

\begin{table}[ht]
\caption{\emph{The polynomials $C_n(q)$}}\label{tableC}
\renewcommand\arraystretch{1.25}
\noindent\[
\begin{array}{|c||c|c|}
\hline
n & C_n(q) & C_n(-1) \\
\hline\hline
1 & q^2 - 2q + 1 & 4 \\ 
\hline
2 & q^4 - q^3 - q + 1 & 4  \\
\hline
3 & q^6 - q^5 - q^4 + 2 q^3 - q^2 - q + 1 &  0  \\ 
\hline
4 & q^8 - q^7- q + 1 & 4  \\ 
\hline
5 & q^{10} - q^9 - q^7+ q^6 +  q^4 - q^3 - q + 1 & 8   \\ 
\hline
6 & q^{12} - q^{11} + q^7 - 2 q^6 + q^5 - q + 1 &  0 \\ 
\hline
7 & q^{14} - q^{13} - q^{10} + q^9 + q^5 - q^4 - q + 1  & 0  \\ 
\hline
8 &  q^{16} - q^{15} - q + 1 &  4 \\ 
\hline
9 &  q^{18} - q^{17} - q^{13} + q^{12} + q^{11} -  q^{10} -  q^8 + q^7+ q^6 - q^5 - q + 1  & 4 \\
\hline
10 &  q^{20} - q^{19} - q^{11} + 2q^{10} -  q^9 - q + 1  & 8 \\
\hline
11 &  q^{22} - q^{21} - q^{16} + q^{15} + q^7 - q^6 - q + 1  & 0 \\
\hline
12 &  q^{24} - q^{23} + q^{15} - q^{14} -  q^{10} + q^9 - q + 1  & 0 \\
\hline
\end{array}
\]
\end{table}

{\footnotesize
\begin{table}[ht]
\caption{\emph{The polynomials $P_n(q)$}}\label{tableP}
\renewcommand\arraystretch{1.3}
\noindent\[
\begin{array}{|c||c|c|c|c|c|c|}
\hline
n & P_n(q) & P_n(1) & P_n(-1) & |P_n(j)| & |P_n(\ii)| & a_{n,0}\\
\hline\hline
1 & 1 & 1 & 1 & 1 & 1 & 1 \\ 
\hline
2 & q^2 + q + 1 & 3 & 1 & 0 & 1 & 1 \\
\hline
3 & q^4 + q^3 + q + 1 & 4 & 0 & 2 & 2 & 0 \\ 
\hline
4 & q^6 + q^5+ q^4 + q^3 + q^2 + q + 1 & 7 & 1 & 1 & 1 & 1 \\ 
\hline
5 & q^8 + q^7+ q^6 + q^2 + q + 1 & 6 & 2 & 0 & 0 & 0 \\ 
\hline
 & q^{10} + q^9 + q^8 + q^7+ q^6  &  &  &  &  & \\ 
6 & + 2q^5+ q^4 + q^3 + q^2 + q + 1& 12 & 0 & 0 & 2 & 2 \\ 
\hline
7 &  q^{12} + q^{11} + q^{10} + q^9 + q^3 + q^2 + q + 1 & 8 & 0 & 2 & 0 & 0 \\ 
\hline
 &  q^{14} + q^{13} + q^{12} + q^{11} + q^{10} + q^9 + q^8   &  &  &  & & \\ 
8 &  + q^7 + q^6 + q^5+ q^4 + q^3 + q^2 + q + 1 & 15 & 1 & 0 & 1 & 1\\ 
\hline
 &  q^{16} + q^{15}+ q^{14} + q^{13} + q^{12} + q^9    &  & &  &  & \\ 
9 &  + q^8 + q^7+ q^4 + q^3 + q^2 + q + 1  & 13 & 1 & 2 & 3 & 1 \\
\hline
 &  q^{18} + q^{17}+ q^{16} + q^{15}+ q^{14} + q^{13}   &  & &  &  & \\ 
  &  + q^{12} + q^{11}  + q^{10} + q^8 + q^7+ q^6   &  & &  &  & \\
10 &   + q^5 + q^4 + q^3 + q^2 + q + 1  & 18 & 2 & 0 & 0 & 0 \\
\hline
 &  q^{20} + q^{19}+ q^{18} + q^{17}+ q^{16} + q^{15}    &  & &  &  & \\ 
11 &   + q^5 + q^4 + q^3 + q^2 + q + 1  & 12 & 0 & 0 & 2 & 0 \\
\hline
 &  q^{22} + q^{21}+ q^{20} + q^{19}+ q^{18} + q^{17}  + q^{16} + q^{15} &  & &  &  & \\ 
 &  + q^{14} + 2 q^{13} + 2 q^{12} +  2q^{11}  + 2q^{10}  + 2 q^9 + q^8  &  & &  &  & \\ 
12 &   + q^7+ q^6 + q^5 + q^4 + q^3 + q^2 + q + 1  & 28 & 0 & 2 & 2 & 2 \\
\hline
\end{array}
\]
\end{table}
}

The present paper, which complements\,\cite{KRidealcount}, is organized as follows.
In Section\,\ref{sec-applic} we compute the local zeta function of
the Hilbert scheme of $n$~points on a two-dimensional torus.
Section\,\ref{sec-proofs} is devoted to the proofs of Theorems\,\ref{th-cni},\,\ref{th-ani} and\,\ref{th-omega}.
In Section\,\ref{sec-misc} we collect additional results on the polynomials~$C_n(q)$ and $P_n(q)$.
Finally, in Appendix\,\ref{sec-Somos} we give a proof of a result by Somos needed in Section\,\ref{ssec-pf-omega}
for the proof of Theorem\,\ref{th-omega}.

\section{The local zeta function of the Hilbert scheme of $n$~points on a two-dimensional torus}\label{sec-applic}

Let $k$ be a field and $n$ be a positive integer.
The algebra~$k[x,y,x^{-1},y^{-1}]$ of Laurent polynomials with coefficients in~$k$ 
is the coordinate ring of the two-dimensional torus $(\AA^1_k \setminus \{0\}) \times (\AA^1_k \setminus \{0\})$.
As is well known, the ideals of codimension~$n$ of the Laurent polynomial algebra~$k[x,y,x^{-1},y^{-1}]$
are in bijection with the $k$-points of the Hilbert scheme
\[
H^n_k = \Hilb^n \left( \left( \AA^1_k \setminus \{0\} \right) \times \left(\AA^1_k \setminus \{0\} \right) \right)
\]
parametrizing finite subschemes of co\-length~$n$ of the two-dimensional torus.
This scheme is a quasi-projective variety. 

As an application of Theorem\,\ref{th-cni}, we now give an explicit expression for the 
local zeta function of the Hilbert scheme~$H^n_{\FF_q}$. Since the scheme is quasi-projective, 
its local zeta function is by\,\cite{Dw, Gr} a rational function.

Recall that the \emph{local zeta function} of an algebraic variety~$X$ defined over a finite field~$\FF_q$ of cardinality~$q$
is the formal power series
\begin{equation}\label{def-Z}
Z_{X/\FF_q}(t) = \exp\left(\sum_{m\geq 1}\, |X(\FF_{q^m})| \, \frac{t^m}{m} \right),
\end{equation}
where $|X(\FF_{q^m})|$ is the number of points of~$X$ over the degree~$m$ extension~$\FF_{q^m}$ of~$\FF_q$.

\begin{theorem}\label{th-mainZ}
The local zeta function of the Hilbert scheme~$H^n_{\FF_q}$ is the rational function
\begin{equation*}
Z_{H^n_{\FF_q}/\FF_q}(t) = 
\frac{1}{(1-q^nt)^{c_{n,0}}} \, \prod_{i=1}^n \, \frac{1}{[(1-q^{n+i}t)(1-q^{n-i}t)]^{c_{n,i}}} \, ,
\end{equation*}
where the exponents $c_{n,i}$ are the integers determined in Theorem\,\ref{th-cni}.
\end{theorem}

Let us display a few examples of such rational functions.
For $n = 3$, $5$,~$6$, we have
\begin{equation*}
Z_{H^3_{\FF_q}/\FF_q}(t) = 
\frac{(1-qt)(1-q^2t)(1-q^4t)(1-q^5t)}{(1-t)(1-q^3t)^2(1-q^6t)} \, ,
\end{equation*}
\begin{equation*}
Z_{H^5_{\FF_q}/\FF_q}(t) = 
\frac{(1-qt)(1-q^3t)(1-q^7t)(1-q^9t)}{(1-t)(1-q^4t)(1-q^6t)(1-q^{10}t)} \, ,
\end{equation*}
\begin{equation*}
Z_{H^6_{\FF_q}/\FF_q}(t) = 
\frac{(1-qt)(1-q^6t)^2(1-q^{11}t)}{(1-t)(1-q^5t)(1-q^7t)(1-q^{12}t)} \, .
\end{equation*}

\begin{proof}[Proof of Theorem\,\ref{th-mainZ}]
Let $Z(t)$ be the RHS of the equality in the theorem. 
Clearly, $Z(0) = 1$. By\,\eqref{def-Z} it remains to check that
\begin{equation*}
t\frac{Z'(t)}{Z(t)} = \sum_{m\geq 1}\, C_n(q^m) t^m.
\end{equation*}
Now 
\begin{eqnarray*}
t\frac{Z'(t)}{Z(t)} 
& = & c_{n,0} \frac{q^nt}{1-q^nt} + \sum_{i=1}^n \, c_{n,i} \left(\frac{q^{n+i}t}{1-q^{n+i}t} + \frac{q^{n-i}t}{1-q^{n-i}t}\right) \\
& = & \sum_{m\geq 1} \, c_{n,0}\, q^{mn}t^m +  
\sum_{i=1}^n \, c_{n,i} \sum_{m\geq 1} \, \left( q^{m(n+i)} + q^{m(n-i)} \right) t^m\\
& = &\sum_{m\geq 1}\, C_n(q^m) t^m.
\end{eqnarray*}
\end{proof}

As one immediately sees, the formula for~$Z_{H^n_{\FF_q}/\FF_q}(t)$ has a striking symmetry, which we express as follows.

\begin{coro}\label{coro-Z}
The local zeta function~$Z_{H^n_{\FF_q}/\FF_q}(t)$ satisfies the functional equation
\begin{equation*}
Z_{H^n_{\FF_q}/\FF_q}\left( \frac{1}{q^{2n}t} \right) = Z_{H^n_{\FF_q}/\FF_q}(t).
\end{equation*}
\end{coro}

Such functional equations exist for smooth projective schemes (see for instance \cite[(2.6)]{De}).
Here we obtained one despite the fact that the Hilbert scheme~$H^n$ is not projective.

\begin{proof}
Using the expression for $Z_{H^n/\FF_q}(t)$ in Theorem\,\ref{th-mainZ}, we have
\begin{eqnarray*}
Z_{H^n/\FF_q}\left( \frac{1}{q^{2n}t} \right)
& = & 
\frac{1}{(1-q^{-n}t^{-1})^{c_{n,0}}} \, \prod_{i=1}^n \, \frac{1}{[(1-q^{-n+i}t^{-1})(1-q^{-n-i}t^{-1})]^{c_{n,i}}} \\
& = & 
\frac{(q^nt)^{c_{n,0}}}{(q^nt-1)^{c_{n,0}}} \, 
\prod_{i=1}^n \, \frac{(q^{n-i}t)^{c_{n,i}}(q^{n+i}t)^{c_{n,i}}}{[(q^{n-i}t-1)(q^{n+i}t-1)]^{c_{n,i}}} \\
& = & (-1)^{c_{n,0}}
\frac{(q^nt)^{c_{n,0}+ 2\sum_{i=1}^n  c_{n,i}}}{(1-q^nt)^{c_{n,0}}} \, \prod_{i=1}^n \, \frac{1}{[(1-q^{n-i}t)(1-q^{n+i}t)]^{c_{n,i}}} \\
& = & (-1)^{c_{n,0}}
\frac{(q^nt)^{C_n(1)}}{(1-q^nt)^{c_{n,0}}} \, \prod_{i=1}^n \, \frac{1}{[(1-q^{n+i}t)(1-q^{n-i}t)]^{c_{n,i}}} \, .
\end{eqnarray*}
This latter is equal to $Z_{H^n/\FF_q}(t)$ in view of the vanishing of~$C_n(1)$ and of the fact that all integers~$c_{n,0}$
are even (see Theorem\,\ref{th-cni}\,(a)).
\end{proof}

With Theorem~\ref{th-mainZ} we can easily compute the \emph{Hasse--Weil zeta function} of the above Hilbert scheme; 
this zeta function is defined by
\begin{equation}\label{def-HWZ}
\zeta_{H^n}(s) = \prod_{p \; \text{prime}} \, Z_{{H^n_{\FF_p}}/\FF_p}(p^{-s}),
\end{equation}
where the product is taken over all prime integers~$p$.
It follows from Theorem~\ref{th-mainZ} that the Hasse--Weil zeta function of~$H^n$ is given by
\begin{equation*}
\zeta_{H^n}(s) = \zeta(s-n)^{c_{n,0}} \prod_{i=1}^n \,   \left[ \zeta(s-n-i) \zeta(s-n+i) \right]^{c_{n,i}},
\end{equation*}
where $\zeta(s)$ is Riemann's zeta function.
As a consequence of \eqref{def-HWZ} and of Corollary\,\ref{coro-Z},  
the Hasse--Weil zeta function $\zeta_{H^n}(s)$ satisfies the functional equation
\begin{equation*}
\zeta_{H^n}(s) = \zeta_{H^n}(2n-s).
\end{equation*}

\section{Proofs}\label{sec-proofs}

We now give the proofs of the results stated in the introduction.
We start with the proof of Theorem\,\ref{th-ani}.

\subsection{Proof of Theorem\,\ref{th-ani}}\label{ssec-pf-ani}

We first establish the positivity of the coefficients of~$P_n(q)$, then we compute them.

(a) \emph{(Positivity)}
Replacing first $C_n(q)$ by $(q-1)^2 P_n(q)$, 
then the variables~$q$ by~$q^2$ and $t$ by~$t^2$ in\,\eqref{gf-Cnq}, we obtain
\begin{equation*}
1 + (q^2-1)^2 \, \sum_{n\geq 1} \, \frac{P_n(q^2)}{q^{2n}}  \, t^{2n}
= \prod_{n \geq 1}\, \frac{(1-t^{2n})^2}{(1-q^2t^{2n})(1-q^{-2} t^{2n})}\,  .
\end{equation*}
By \cite[Eq.\,(10.1)]{Fi}, 
\begin{multline*}
\frac{2q}{1-q^2}\, \prod_{n \geq 1}\, \frac{(1-t^{2n})^2}{(1-q^2t^{2n})(1-q^{-2} t^{2n})} \\
= \frac{2q}{1-q^2} + \sum_{k,m \geq 1} \, \left( (-1)^m - (-1)^k \right) q^{k-m} t^{km}.
\end{multline*}
Therefore,
\begin{equation*}
1 + (q^2-1)^2 \, \sum_{n\geq 1} \, \frac{P_n(q^2)}{q^{2n}}  \, t^{2n}
= 1 + \frac{1-q^2}{2q}  \sum_{k,m \geq 1} \, \left( (-1)^m - (-1)^k \right) q^{k-m} \, t^{km} .
\end{equation*}
Equating the coefficients of~$t^{2n}$ for $n\geq 1$, we obtain
\begin{equation*}
\frac{P_n(q^2)}{q^{2n}} = \frac{1}{2q(1-q^2)} \sum_{k,m \geq 1 \atop km = 2n} \, \left( (-1)^m - (-1)^k \right) q^{k-m}.
\end{equation*}
Now the integer $(-1)^m - (-1)^k$ vanishes if $m, k$ are of the same parity, 
is equal to~$2$ if $m$ is even and $k$ is odd, and is equal to~$-2$ if $m$ is odd and $k$ is even.
In view of this remark, we have
\begin{eqnarray*}
P_n(q^2) 
& = & \frac{q^{2n-1}}{2(1-q^2)} \sum_{1 \leq k < m \atop km = 2n} \, 
\left( \left( (-1)^m - (-1)^k \right) q^{k-m} + \left( (-1)^k - (-1)^m \right) q^{m-k}\right) \\
& = & \frac{1}{2(1-q^2)} \sum_{1 \leq k < m \atop km = 2n} \, \left( (-1)^m - (-1)^k \right) q^{2n-1} q^{k-m} (1 - q^{2m - 2k}) \\
& = & \sum_{1 \leq k < m \atop km = 2n} \, \frac{(-1)^m - (-1)^k}{2}  \, q^{2n-1+k-m}  \frac{1 - q^{2m - 2k}}{1-q^2} \\
& = & \sum_{1 \leq k < m \atop km = 2n} \, \frac{(-1)^m - (-1)^k}{2}  \, q^{2n-1+k-m}  
\left( \sum_{j=0}^{m-k-1} \, q^{2j} \right) .
\end{eqnarray*} 
We now restrict this sum to the pairs $(k,m)$ of different parity, in which case $((-1)^m - (-1)^k)/2 = (-1)^{k-1}$.
Using the notation $H(I) = \sum_{2k \in I} \, q^{2k}$ for any interval $I \subset \NN$, we obtain
\begin{equation}\label{eq-positive}
P_n(q^2) = \sum_{1 \leq k < m, \; km = 2n \atop k\not\equiv m \!\!\pmod 2} \, (-1)^{k-1} \, H\left( [2n-1+k-m, 2n-3 + m -k] \right) .
\end{equation}
Note that the intervals are all centered around~$0$ and their bounds are even.

Now the only contributions in the previous sum which may pose a positivity problem occur when $k$ is even; 
so in this case write $k = 2^{\alpha} k'$ for some integer~$\alpha \geq 1$ and some odd integer $k'\geq 1$.
Setting $m' = 2^{\alpha} m$, we see that $k'$ is odd, $m'$ is even, $k'm' = km = 2n$ and $1 \leq k' < k < m < m'$.
Moreover, 
\begin{equation*}
[2n-1+k-m, 2n-3 + m -k] \subsetneq [2n-1+k'-m', 2n-3 + m' -k'] .
\end{equation*}
Observe that the function $k \mapsto k' = 2^{-\alpha}k$ is injective. 
It follows that the negative contributions in the sum\,\eqref{eq-positive} corresponding to even~$k$ are counterbalanced by
the positive contributions corresponding to odd~$k'$. This shows that $P_n(q^2)$, hence~$P_n(q)$,
has only non-negative coefficients.

(b) \emph{(Computation of the coefficients of~$P_n(q)$)}
Using\,\eqref{eq-positive}, we immediately derive
\begin{equation}\label{eq-positive2}
\frac{P_n(q)}{q^{n-1}} = 
\sum_{1 \leq k < m, \; km = 2n \atop k\not\equiv m \!\!\pmod 2} \, (-1)^{k-1} \, \left[\left[-\frac{m -k - 1}{2}, \frac{m -k - 1}{2} \right]\right] ,
\end{equation}
where $[[a,b]]$ stands for $q^a + q^{a+1} + \cdots + q^b$ (for two integers $a \leq b$).
Observe that $m-k-1$ is even so that the bounds in the RHS of\,\eqref{eq-positive2} are integers.
It follows from this expression of~$P_n(q)$ that $q^{-(n-1)}P_n(q)$ is invariant under the map $q \mapsto q^{-1}$.
This shows that the degree~$2n-2$ polynomial~$P_n(q)$ is palindromic 
(which we had already pointed out in the introduction as the infinite product 
in\,\eqref{gf-Cnq} is invariant under the transformation $q \mapsto q^{-1}$).
Therefore, 
\begin{equation}\label{eq-P-a}
\frac{P_n(q)}{q^{n-1}} = a_{n,0} + \sum_{i=1}^{n-1} \, a_{n,i} \, (q^i + q^{-i})
\end{equation}
for some non-negative integers~$a_{n,i}$, which we now determine. 
We assert the following.

\begin{prop}\label{prop-gf-a}
For any integer $i\geq 0$, we have
\begin{equation*}
\sum_{n\geq 1} \, a_{n,i} \, t^n =
\sum_{k\geq 1} \, (-1)^{k-1}\, \frac{t^{k(k+1)/2} \, t^{ki}}{1 - t^k} \, .
\end{equation*}
\end{prop}

\begin{proof}
The coefficient $a_{n,i}$ is the coefficient of~$q^i$ in~$q^{-(n-1)}P_n(q)$. 
It follows from \eqref{eq-positive2} that 
\begin{equation*}
a_{n,i} = \sum_{A(n,i)}\, (-1)^{k-1},
\end{equation*}
where $A(n,i)$ is the set of integers $k$ such that $km = 2n$, $1 \leq k < m$, $k\not\equiv m \pmod 2$, and 
$(m-k-1)/2 \geq i$ (the latter condition is equivalent to $m-k \geq 2i+1$). 
Now the RHS in the proposition is equal to
\begin{equation*}
\sum_{k\geq 1} \, (-1)^{k-1}\, t^{k(k+1)/2} \, t^{ki} \sum_{j\geq 0}\, t^{kj}
= \sum_{k\geq 1} \, (-1)^{k-1}\, \sum_{j\geq 0}\, t^{k(k+2i+2j +1)/2}.
\end{equation*}
Thus the coefficient of $q ^i$ in the latter expression is equal to
\begin{equation*}
\sum_{k\geq 1, \, j\geq 0\atop k(k+2i+2j+1)= 2n} \, (-1)^{k-1} .
\end{equation*}
Setting $m = k+2i+2j+1$, we see that the previous sum is the same as the one indexed by the set~$A(n,i)$ above.
\end{proof}

Note the following consequence of \eqref{eq-P-a} and of Proposition~\ref{prop-gf-a}.

\begin{coro}
We have
\begin{equation*}
\sum_{n\geq 1} \, \frac{P_n(q)}{q^{n-1}} \, t^n =
\sum_{k\geq 1} \, (-1)^{k-1}\, t^{k(k+1)/2} \,\frac{1 + t^k}{(1 - qt^k)(1 - q^{-1}t^k)} \, .
\end{equation*}
\end{coro}

To compute the coefficients~$a_{n,i}$, we adapt the proof in\,\cite{CESM}.
Using Proposition\,\ref{prop-gf-a} and separating the even indices~$k$ from the odd ones, we have
\begin{eqnarray*}
\sum_{n\geq 1} \, a_{n,i} \, t^n 
& = & 
\sum_{k\; \text{odd}\, \geq 1} \,  \frac{t^{k(k+2i+1)/2} }{1 - t^k} 
- \sum_{k\geq 1} \,  \frac{t^{k(2k+2i+1)} }{1 - t^{2k}} \\
& = & \sum_{k\; \text{odd}\, \geq 1} \,  \frac{t^{k(k+2i+1)/2} }{1 - t^k}
- \sum_{k\geq 1} \,  \frac{t^{k(2k+2i+1)} (1+t^k)}{1 - t^{2k}} 
+  \sum_{k\geq 1} \,  \frac{t^{k(2k+2i+2)} }{1 - t^{2k}}\\
& = &\sum_{k\; \text{odd}\, \geq 1} \,  \frac{t^{k(k+2i+1)/2} }{1 - t^k} 
+  \sum_{k\; \text{even}\, \geq 1} \,  \frac{t^{k(k+2i+2)/2} }{1 - t^k}
- \sum_{k\geq 1} \,  \frac{t^{k(2k+2i+1)}}{1 - t^k} \\
& = &   \sum_{k\geq 1} \, \frac{t^{k(\lfloor k/2 \rfloor+i+1)} }{1 - t^k} 
- \sum_{k\geq 1} \,  \frac{t^{k(2k+2i+1)}}{1 - t^k}
\end{eqnarray*}
since
\begin{equation*}
\lfloor k/2 \rfloor+i+1 = 
\begin{cases}
(k+2i+1)/2 \quad \text{if} \; k \; \text{is odd,} \\
(k+2i+2)/2 \quad \text{if} \; k \; \text{is even.}
\end{cases}
\end{equation*}

Using the inequality $\lfloor k/2 \rfloor+i+1 < 2k+2i+1$ and the identity
\begin{equation*}
\frac{t^{ka}}{1-t^k} - \frac{t^{kb}}{1-t^k}
= \sum_{a \leq d <b} \, t^{kd} \qquad\qquad (a < b)
\end{equation*}
in the special case $a = \lfloor k/2 \rfloor+i+1$ and $b = 2k+2i+1$, we obtain
\begin{equation*}
\sum_{n\geq 1} \, a_{n,i} \, t^n 
= \sum_{k\geq 1}\, \sum_{\lfloor k/2 \rfloor+i+1 \leq d < 2k+2i+1} \,  t^{kd} .
\end{equation*}
Therefore, $a_{n,i}$ is the number of divisors~$d$ of~$n$ such that 
\begin{equation*}
\left\lfloor \frac{n}{2d} \right\rfloor+i+1 \leq d <  \frac{2n}{d}+2i+1.
\end{equation*}
The leftmost inequality is equivalent to $n/(2d) + i < d$, which is equivalent to $d^2 - id -n/2 > 0$,
which in turn is equivalent to $d> (i + \sqrt{2n+i^2})/2$. 
On the other hand, the rightmost inequality is equivalent to $d \leq 2n/d + 2i$, 
which is equivalent to $d^2 - 2id -2n \leq 0$, which in turn is equivalent to $d \leq i + \sqrt{2n+i^2}$.

This completes the proof of Theorem\,\ref{th-ani}.

\subsection{Proof of Theorem\,\ref{th-cni}}\label{ssec-pf-cni}

Let $c_{n,i}$ be the coefficients of~$C_n(q)$ as defined by\,\eqref{def-cni}.
We claim the following.

\begin{prop}\label{prop-gf-c}
We have
\begin{equation*}
\sum_{n\geq 1} \, c_{n,0} \, t^n 
= 2 \sum_{k\geq 1} \, (-1)^k\, t^{k(k+1)/2} ,
\end{equation*}
and for any integer $i\geq 1$, we have
\begin{equation*}
\sum_{n\geq 1} \, c_{n,i} \, t^n =
\sum_{k\geq 1} \, (-1)^k \, \left( t^{k(k+2i+1)/2}  - t^{k(k+2i-1)/2} \right) .
\end{equation*}
\end{prop}

\begin{proof}
The equality $C_n(q)/q^n = (q - 2 + q^{-1}) P_n(q)/q^{n-1}$ implies that the coefficients $c_{n,i}$ of~$C_n(q)$ are
related to the coefficients $a_{n,i}$ of~$P_n(q)$ by the relations
\begin{equation}\label{eq-c1}
c_{n,0} = - 2 a_{n,0} + 2 a_{n,1} 
\end{equation}
and 
\begin{equation}\label{eq-c2}
c_{n,i} = a_{n,i+1} - 2 a_{n,i} +  a_{n,i-1} \qquad (i\geq 1).
\end{equation}

From\,\eqref{eq-c1} and Proposition\,\ref{prop-gf-a}, we obtain
\begin{equation*}
\sum_{n\geq 1} \, c_{n,0} \, t^n 
= 2 \sum_{k\geq 1} \, (-1)^{k-1}\, \frac{t^{k(k+1)/2} \, (t^{k}-1)}{(1 - t^k)} 
= 2 \sum_{k\geq 1} \, (-1)^k\, t^{k(k+1)/2} .
\end{equation*}
If $i\geq 1$, then by \,\eqref{eq-c2} and Proposition\,\ref{prop-gf-a},
\begin{eqnarray*}
\sum_{n\geq 1} \, c_{n,i} \, t^n 
& = &  \sum_{k\geq 1} \, (-1)^{k-1}\, \frac{t^{k(k+1)/2} \, (t^{k(i+1)} - 2 t^{ki} + t^{k(i-1)})}{(1 - t^k)} \\
& = & \sum_{k\geq 1} \, (-1)^{k-1}\, \frac{t^{k(k+1)/2} \, t^{k(i-1)} (1 - t^k)^2}{(1 - t^k)} \\
& = & \sum_{k\geq 1} \, (-1)^{k-1}\, (1 - t^k)\, t^{k(k+2i-1)/2} \\
& = & \sum_{k\geq 1} \, (-1)^k \left( t^{k(k+2i +1)/2} - t^{k(k+2i-1)/2} \right).
\end{eqnarray*}
\end{proof}

It follows from the first equality in Proposition\,\ref{prop-gf-c}
that $c_{n,0} = 0$ unless $n$ is \emph{triangular}, i.\,e., of the form $n= k(k+1)/2$, in which case $c_{n,0} = 2\, (-1)^k$.
This proves Theorem\,\ref{th-cni}\,(a).

For Part\,(b) of the theorem we introduce the following definition.

\begin{definition}
Fix an integer $i\geq 0$. An integer $n > 0$ is $i$-trapezoidal if there is an integer $k \geq 1$ such that
\begin{equation*}
n = (i+1) + (i+2) + \cdots + (i+k) = \frac{k(k+2i+1)}{2} \, .
\end{equation*}
\end{definition}

A $0$-trapezoidal number is a triangular number. 

\begin{lemma}\label{lem-trapezoidal}
(1) An integer~$n$ is $i$-trapezoidal if and only if $8n + (2i+1)^2$ is a perfect square.

(2) If $n$ is $i$-trapezoidal for some $i\geq 1$, then it is not $(i-1)$-trapezoidal.
\end{lemma}

\begin{proof}
(1) An integer~$n$ is $i$-trapezoidal if and only if $x^2 + (2i+1)x -2n=0$ has an integral solution~$k \geq 1$.
This implies that the discriminant $8n + (2i+1)^2$ is a perfect square. Conversely, if the discriminant is the square
of a positive integer~$\delta$, then $\delta$ must be odd, and $k= (\delta -(2i+1))/2$ is a positive integer.

(2) Consider the sequence $((2(i+j) +1)^2)_{j\geq 0}$ of odd perfect squares $\geq (2i+1)^2$. 
The distance between the square $(2(i+j) +1)^2$ and the subsequent one is equal to~$8(i+j+1)$, 
which, if $j\geq 0$, is strictly bigger 
than the distance~$8i$  between $(2i-1)^2$ and $(2i+1)^2$. 
Now suppose that $n$ is both $i$-trapezoidal and $(i-1)$-trapezoidal. 
By Part\,(1), both $8n + (2i+1)^2$ and $8n + (2i-1)^2$ are odd perfect squares,
which are $> (2i-1)^2$ since $n\geq 1$. Hence they are $\geq (2i+1)^2$. 
Since their difference is~$8i$, we obtain a contradiction.
\end{proof}

We now complete the proof of Theorem\,\ref{th-cni}\,(b). 
Since by Lemma\,\ref{lem-trapezoidal} an integer $n$ cannot be $i$-trapezoidal for two consecutive~$i$, 
it follows from the second equality in Proposition\,\ref{prop-gf-c} that
\begin{equation*}
c_{n,i} =
\left\{
\begin{array}{cl}
(-1)^k & \text{if}\; n = k(k+2i +1)/2 \; \text{for some integer}\; k, \\
\noalign{\smallskip}
(-1)^{k-1} & \text{if}\; n = k(k+2i -1)/2 \; \text{for some integer}\; k, \\
\noalign{\smallskip}
0 & \text{otherwise}.
\end{array}
\right.
\end{equation*}

\begin{rem}
By Theorem\,\ref{th-cni} we have the inequalities $|c_{n,i}| \leq 2$.
Together with \eqref{eq-c1} and\,\eqref{eq-c2}, they imply that 
the variation of the coefficients~$a_{n,i}$ of~$P_n(q)$ is small. More precisely,
\begin{equation*}\label{variation}
\left| \, a_{n,i} - \frac{a_{n,i-1} + a_{n,i+1}}{2} \, \right| \leq 1
\end{equation*}
for all $n \geq 1$ and all $i\geq 0$ (we use here the convention $a_{n,-1} = a_{n,1}$).
\end{rem}

\subsection{Proof of Theorem\,\ref{th-omega}}\label{ssec-pf-omega}

Let $\omega$ be a complex root of unity of order $d = 2$, $3$, $4$, or~$6$.
For each such~$d$ we define the sequence~$(a_d(n))_{n\geq 1}$ by
\begin{equation}\label{def-adn}
1 + \sum_{n\geq 1} \, a_d(n)  \, t^n
= \prod_{i\geq 1}\, \frac{(1-t^i)^2}{1-(\omega + \omega^{-1})t^i + t^{2i}} \, .
\end{equation}
As observed in the introduction each $a_d(n)$ is an integer.
In view of\,\eqref{gf-Cnq} and of the definition of~$P_n(q)$ we have
\begin{equation*}
a_d(n) = \frac{C_n(\omega)}{\omega^n} = \left(\omega + \omega^{-1} -2 \right)  \frac{P_n(\omega)}{\omega^{n-1}} \, .
\end{equation*}
In order to prove Theorem\,\ref{th-omega} we compute $a_2(n)$, $a_3(n)$, $a_4(n)$, and $a_6(n)$ successively.
See Table\,\ref{6th-root} for the absolute values~$|a_d(n)|$ with $1 \leq n\leq 18$.

\subsubsection{The case $d=2$}

Setting $\omega = -1$ in\,\eqref{def-adn}, we obtain
\begin{equation}\label{def-a2n}
1 + \sum_{n\geq 1} \, a_2(n) \, t^n
= \prod_{i\geq 1}\, \frac{(1-t^i)^2}{1 + 2t^i +  t^{2i}}
= \left( \prod_{i\geq 1}\, \frac{1-t^i}{1+t^i}\right)^2.
\end{equation}
Now recall the following identity of Gauss (see\,\cite[(7.324)]{Fi} or\,\cite[19.9\,(i)]{HW}):
\begin{equation}\label{eq-Gauss}
\prod_{i\geq 1}\, \frac{1-t^i}{1 + t^i} = \sum_{k\in \ZZ}\, (-1)^k t^{k^2}.
\end{equation}
It follows from this identity that the RHS of\,\eqref{def-a2n} is equal to
\begin{equation*}
\left( \sum_{k\in \ZZ}\, (-1)^k t^{k^2} \right)^2 .
\end{equation*}
Since $k$ and $k^2$ are of the same parity, we obtain
\begin{equation*}
1 + \sum_{n\geq 1} \, a_2(n) \, t^n
= \left( \sum_{k\in \ZZ}\, (-t)^{k^2} \right)^2 .
\end{equation*}
Now the latter is clearly equal to $\sum_{n\geq 0} \, r(n) (-t)^n$, 
where $r(n)$ is the positive integer defined by\,\eqref{def-r1n}.
Identifying the terms, we deduce $a_2(n) = (-1)^n r(n)$ for all $n\geq 1$.
Since $a_2(n) =  (-1)^n C_n(-1) = (-1)^n 4 P_n(-1)$, we deduce the 
desired values for~$C_n(-1)$ and~$P_n(-1)$.

\subsubsection{The case $d=3$}

Let $j$ be a primitive third root of unity.
Setting $\omega = j$ in\,\eqref{def-adn}, we obtain
\begin{equation*}
1 + \sum_{n\geq 1} \, a_3(n)  \, t^n
= \prod_{i\geq 1}\, \frac{(1-t^i)^2}{1 + t^i + t^{2i}} =  \prod_{i\geq 1}\, \frac{(1-t^i)^3}{1 - t^{3i}} \, .
\end{equation*}
By\,\cite[\S\,32, p.\,79]{Fi},
\begin{equation*}\label{eq-a3}
a_3(n) =  - 3 \lambda(n) = -3 \left( E_1(n;3) - 3 E_1(n/3;3) \right)
\end{equation*}
where $\lambda(n)$ is the multiplicative function of~$n$ defined by\,\eqref{def-lambda}
and $E_1(n;3)$ is the excess function\,\eqref{def-excess}.
Hence, the desired values for~$C_n(j)$ and~$P_n(j)$.

\begin{rem}
By\,\cite{Sl}, the series $\sum_{n\geq 0} \, 3 \lambda(n) t^n$ is the theta series (with respect to a node)
of the two-dimensional honeycomb (not a lattice) in which each node has three neighbors. 
The sequence $a_3(n)$ (resp.~$\lambda(n)$) is Sequence~A005928 (resp. A113063) of\,\cite{OEIS}.

By~\cite[Seq.\,A005928]{OEIS}, $a_3(3n+2) = 0$, and $-a_3(3n+1)$ is Sequence~A005882 of\,\cite{OEIS}.
More interestingly, 
the sequence $a_3(3n) = r''(3n)$, where
\begin{equation}\label{def-r3n}
r''(n) = \left\{ (x,y) \in \ZZ^2 \, |\, x^2 + xy + y^2 =n  \right \}.
\end{equation}
The generating function of the sequence~$r''(n)$ is the theta series of the hexagonal lattice 
in which each point has six neighbors
(see Sequence~A004016 of\,\cite{OEIS} or~\cite[Chap.~4, \S~6.2]{CS}).
\end{rem}

\subsubsection{The case $d=4$} 

Let $\ii = \sqrt{-1}$ be a square root of~$-1$.
It suffices to compute~$a_4(n)$. It follows from\,\eqref{def-adn} that
\begin{equation}\label{def-a4n}
1 + \sum_{n\geq 1} \, a_4(n) \, t^n
= \prod_{i\geq 1}\, \frac{(1-t^i)^2}{1 + t^{2i}} \, .
\end{equation}
Then $a_4(n)$ forms Sequence~A082564 of\,\cite{OEIS} and we have
\begin{equation*}
a_4(n) = (-1)^{\lfloor (n+1)/2\rfloor}  \, |a_4(n)|.
\end{equation*}
By Lemma~\ref{lem-app} of the appendix, for the absolute value of~$a_4(n)$ we have
\begin{equation*}
1 + \sum\, |a_4(n)| \, q^n  
= \varphi(q) \varphi (q^2)
= \left( \sum_{n\in \ZZ}\, q^{n^2} \right) \left( \sum_{n\in \ZZ}\, q^{2n^2} \right),
\end{equation*}
which implies that $|a_4(n)| = r'(n)$, where $r'(n)$ is the positive integer defined by\,\eqref{def-r2n}.
Therefore, $|a_4(n)|$ forms Sequence~A033715 of\,\cite{OEIS}. 

\subsubsection{The case $d=6$}

Since $C_n(-j) = (-1)^n a_6(n) j^n$ and $P_n(-j) = (-1)^n a_6(n) j^{n-1}$, it suffices to compute~$a_6(n)$.
Setting $\omega = -j$ (which is a primitive sixth root of unity) in \eqref{def-adn}, we obtain
\begin{equation*}
1 + \sum_{n\geq 1} \, a_6(n)  \, t^n
= \prod_{i\geq 1}\, \frac{(1-t^i)^2}{1 - t^i + t^{2i}} =  \prod_{i\geq 1}\, \frac{(1-t^i) (1-t^{2i})(1 - t^{3i})}{1 - t^{6i}} \, .
\end{equation*}
We computed the integers~$a_6(n)$ in\,\cite{KReta} (see Theorem\,1.1 there);
they form Sequence A258210 in\,\cite{OEIS}.
This completes the proof of Theorem\,\ref{th-omega}.

\begin{rem}
Our result for~$a_6(n)$ shows that $a_6(n) = 0$ if and only if $a_2(n) = 0$, 
i.e.\ if and only $n$ is not the sum of two squares.
\end{rem}

\begin{table}[ht]
\caption{\emph{The absolute values of $a_d(n)$}}\label{6th-root}
\renewcommand\arraystretch{1.2}
\noindent\[
\begin{array}{|c||c|c|c|c|c|c|c|c|c|c|c|c|c|c|c|c|c|c|}
\hline
n & 1 & 2 & 3 & 4 & 5 & 6 & 7& 8& 9 & 10 & 11 & 12 & 13 & 14& 15 & 16 & 17 & 18\\
\hline
\hline
|a_2(n)| & 4 & 4 & 0 & 4 & 8 & 0 & 0 & 4 & 4 & 8 & 0 & 0 & 8 & 0 & 0 & 4 & 8 & 4 \\
\hline
|a_3(n)| & 3 & 0 & 6 & 3 & 0 & 0 & 6 & 0 & 6 & 0 & 0 & 6 & 6 & 0 & 0 & 3 & 0 & 0\\
\hline
|a_4(n)| & 2 & 2 & 4 & 2 & 0 & 4 & 0 & 2 & 6 & 0 & 4 & 4 & 0 & 0 & 0 & 2 & 4 & 6 \\
\hline
|a_6(n)| & 1 & 2 & 0 & 1 & 4 & 0 & 0 & 2 & 4 & 2 & 0 & 0 & 2 & 0 & 0 & 1 & 4 & 4\\
\hline
\end{array}
\]
\end{table}

\section{Growth and sections}\label{sec-misc}

In this section we collect further results on the polynomials~$C_n(q)$ and~$P_n(q)$.

\subsection{Growth of~$C_n(q)$}\label{ssec-growth}

(a) Fix an integer~$n\geq 1$.
Let us examine how $C_n(q)$ grows when $q$ tends to infinity. 

We know that $C_n(q)$ counts the number of $\FF_q$-points of the Hilbert scheme~$H^n_{\FF_q}$
of $n$~points on the two-dimensional torus. This scheme is an open subset of the 
Hilbert scheme $\Hilb^n(\AA_{\FF_q}^2)$ of $n$~points on the affine plane. Let $A_n(q)$ be the 
number of $\FF_q$-points of~$\Hilb^n(\AA_{\FF_q}^2)$; 
by\,\cite[Remark\,4.7]{KRidealcount} $A_n(q)$ is a monic polynomial of degree~$2n$ in~$q$.
Since $C_n(q)$ is also a monic polynomial of degree~$2n$, we necessarily have
\begin{equation*}
C_n(q) \sim A_n(q) \sim q^{2n}.
\end{equation*}
Moreover, the cardinality $A_n(q) - C_n(q)$ of the complement of $H^n_{\FF_q}$ in $\Hilb^n(\AA_{\FF_q}^2)$ 
must become small compared to~$A_n(q)$ when $q$ tends to~$\infty$. 
Indeed, using the expansions
\begin{equation*}
A_n(q) = q^{2n} + q^{2n-1} +  \text{terms of degree} \, \leq 2n-2 
\end{equation*}
and
\begin{eqnarray*}
C_n(q) 
& = & (q-1)^2 \left(q^{2n-2} + q^{2n-3} +  \text{terms of degree} \, \leq 2n-4 \right)\\
&  = & q^{2n} - q^{2n-1} + \text{terms of degree} \, \leq 2n-2,
\end{eqnarray*}
we easily show that
\begin{equation*}
\frac{A_n(q) - C_n(q)}{A_n(q)} = \frac{2}{q} + O\left( \frac{1}{q^2} \right) .
\end{equation*}

(b) We now fix $q$ and let $n$ tend to infinity. Since by Theorem\,\ref{th-cni}
we have $|c_{n,0}| \leq 2$ and $|c_{n,i}| \leq 1$ for $i\neq 0$, 
we obtain the inequality
\[
|C_n(q)| \leq q^n + \sum_{i=0}^{2n}\, q^i = q^n + \frac{q^{2n+1}-1}{q-1} \, ,
\]
which implies that $|C_n(q)|$ is bounded above by a function of~$n$ equivalent to 
\[
\frac{1}{q-1} \, q^{2n+1}.
\]
Hence, $|P_n(q)|$ is bounded above by a function equivalent to $q^{2n+1}/(q-1)^3$.

\subsection{Sections of the polynomials~$P_n(q)$}\label{ssec-sect}

Given an integer $k\geq 1$, we define the \emph{$k$-section}~$s_k(n)$ of~$P_n(q)$ to be the sum
of the (positive) coefficients of the monomials of~$P_n(q)$ of the form~$q^{ki}$ ($i\geq 0$).
We now compute $s_1(n)$, $s_2(n)$, $s_3(n)$, $s_4(n)$, and~$s_6(n)$.
We keep the notation introduced in the introduction.

(a) When $k = 1$, we clearly have $s_1(n) = P_n(1) = \sigma(n)$,
where $\sigma(n)$ is the sum of positive divisors of~$n$.

(b) When $k=2$, then $s_2(n) = (P_n(1) + P_n(-1))/2$, which by Theorem\,\ref{th-omega}\,(a) implies
\[
s_2(n) = \frac{\sigma(n) + r(n)/4}{2} \, .
\]
By\,\cite[Table\,2, Symmetry \textit{p4}]{Ru}, 
$s_2(n)$ is equal to the number of orbits of subgroups of index~$n$ of~$\ZZ^2$ under 
the action of the cyclic group generated by the rotation of angle~$\pi/2$.

(c) For $k=3$, Theorem\,\ref{th-omega}\,(b) implies
\[
s_3(n) = \frac{P_n(1) + P_n(j) + P_n(j^2)}{3} =  \frac{\sigma(n) + \lambda(n) \,(j^{n-1} + j^{-(n-1)})}{3} \, .
\]
By\,\cite[Table\,2, Symmetry \textit{p6}]{Ru}, 
\[
s_3(n) = \frac{\sigma(n) + r''(n)/3}{3} \, ,
\]
where $r''(n)$ is defined by\,\eqref{def-r3n},
and $r''(n)/6$ counts those subgroups of index~$n$ of the lattice generated by~$1$ and~$j$ in~$\CC$ 
which are fixed under the rotation of angle~$\pi/3$.
According to~\cite[\S~51, p.~80, Exercise~XXII.2]{Di}, we have $r''(n) = 6 \, E_1(n;3)$, 
where $E_1(n;3)$ is the excess function~\eqref{def-excess}.
The integers~$s_3(n)$ form Sequence~A145394 of~\cite{OEIS}.

(d) Let $k=4$. By Theorem\,\ref{th-omega}\,(c) we have
\begin{eqnarray*}
s_4(n) 
& = &  \frac{P_n(1) + P_n(\ii) + P_n(-1) + P_n(-\ii)}{4}\\
& = & \frac{\sigma(n) + r(n)/4 + (-1)^{\lfloor (n-1)/2\rfloor} \, r'(n) \, (\ii^{n-1} + \ii^{-(n-1)})/2}{4} \, .
\end{eqnarray*}

(e) For $k = 6$, by Theorem\,\ref{th-omega}\,(b) and\,(d) we have
\begin{eqnarray*}
s_6(n) & = & \frac{P_n(1) + P_n(-j^2) + P_n(j) + P_n(-1) + P_n(j^2) + P_n(-j)}{6} \\
& = & \begin{cases}
\displaystyle{\frac{\sigma(n) - 3r(n)/4 - \lambda(n)}{6}}  & \text{if} \; n \equiv 0,  \\
\noalign{\medskip}
\displaystyle{\frac{\sigma(n) + 3r(n)/4 + 2 \lambda(n)}{6}} & \text{if} \; n \equiv 1, \quad\pmod{3}\\
\noalign{\medskip}
\displaystyle{\frac{\sigma(n) + 3r(n)/4 - \lambda(n)}{6}} & \text{if} \; n \equiv 2.
\end{cases}
\end{eqnarray*}

The values of~$s_2(n)$, $s_3(n)$, $s_4(n)$ and~$s_6(n)$ for $n\leq 18$ are given in Table\,\ref{4-sect}.

\begin{table}[ht]
\caption{\emph{Sections of the polynomials~$P_n(q)$}}\label{4-sect}
\renewcommand\arraystretch{1.2}
\noindent\[
\begin{array}{|c||c|c|c|c|c|c|c|c|c|c|c|c|c|c|c|c|c|c|c|c|c|c|}
\hline
n & 1 & 2 & 3 & 4 & 5 & 6 & 7& 8& 9 & 10 & 11 & 12 & 13 & 14 & 15 & 16 & 17& 18\\
\hline
\hline
s_2(n) & 1 &  2 & 2 & 4 & 4 & 6 & 4 & 8 & 7 & 10 & 6 & 14 & 7 & 12 & 12 & 16 & 10 & 20 \\
\hline
s_3(n) & 1 & 1 & 2 & 3 & 2 & 4 & 4 & 5 & 5 & 6 & 4 & 10 & 6 & 8 & 8 & 11 & 6 & 13 \\
\hline
s_4(n) & 1 & 1 & 2 & 2 & 2 & 3 & 2 & 4 & 5 & 5 & 4 & 7 & 4 & 6 & 6 & 8 & 6 & 10 \\
\hline
s_6(n) & 1 & 1 & 1 & 2 & 2 & 2 & 2 & 3 & 2 & 4 & 2 & 5 & 4 & 4 & 4 & 6 & 4 & 6 \\
\hline
\end{array}
\]
\end{table}


\appendix

\section{On a result by Michael Somos}\label{sec-Somos}

This appendix is based on the use of multisections of power series, 
which is an idea due to Michael Somos (see\,\cite{So}).

Replacing $t$ by $q$ in\,\eqref{def-a4n}, we have
\begin{equation*}
1 + \sum_{n\geq 1} \, a_4(n) \, q^n = \prod_{i\geq 1}\, \frac{(1-q^i)^2}{1 + q^{2i}} \, .
\end{equation*}
We now express this generating function and the generating function of the absolute values~$|a_4(n)|$
in terms of Ramanujan's $\varphi$-function
\begin{equation*}
\varphi(q) = \sum_{n\in \ZZ}\, q^{n^2} = 1 + 2 \sum_{n\geq 1}\, q^{n^2}. 
\end{equation*}

The following lemma is due to Somos (see\,\cite[A082564 and A033715]{OEIS}).

\begin{lemma}\label{lem-app}
We have
\begin{equation*}
1 + \sum_{n\geq 1} \, a_4(n) \, q^n = \varphi(-q) \varphi(-q^2)
\quad\text{and}\quad
1 +  \sum_{n\geq 1} \, |a_4(n)| \, q^n = \varphi(q) \varphi(q^2) .
\end{equation*}
\end{lemma}

\begin{proof}
Replacing $t$ by $-q$ in Gauss's identity\,\eqref{eq-Gauss}, 
we deduce that $\varphi(q)$ can be expanded as the infinite products
\begin{equation}\label{phi-prod}
\varphi(q) 
= \prod_{n\geq 1}\, \frac{1 - (-1)^n q^n}{1 + (-1)^n q^n} 
= \prod_{n\geq 1}\, \frac{(1+q^{2n-1})(1-q^{2n})}{(1-q^{2n-1})(1+q^{2n})}. 
\end{equation}
Using\,\eqref{phi-prod},
we have
\begin{eqnarray*}
\varphi(-q) \varphi(-q^2)
& = & \prod_{n\geq 1}\, 
\frac{(1-q^{2n-1})(1-q^{2n})(1-q^{4n-2})(1-q^{4n})}{(1+q^{2n-1})(1+q^{2n})(1+q^{4n-2})(1+q^{4n})} \\
& = & \prod_{n\geq 1}\, \frac{(1 - q^{2n})(1 - q^n)}{(1 + q^{2n})(1 + q^n)} \\
& = & \prod_{n\geq 1}\, \frac{(1 - q^{2n})^2(1 - q^n)^2}{(1 + q^{2n})(1 + q^n)(1 - q^{2n})(1 - q^n)} \\
& = & \prod_{n\geq 1}\, \frac{(1 - q^{2n})^2(1 - q^n)^2 }{(1 - q^{4n})(1 - q^{2n})}  \\
& = & \prod_{n\geq 1}\, \frac{(1 - q^{2n})(1 - q^n)^2 }{1 - q^{4n}} \\
& = & \prod_{n\geq 1}\, \frac{(1 - q^n)^2 }{1 + q^{2n}}
= 1 + \sum_{n\geq 1} \, a_4(n) \, q^n.
\end{eqnarray*}
Similarly,
\begin{eqnarray*}
\varphi(q) \varphi(q^2)
& = & \prod_{n\geq 1}\, \frac{(1+q^{2n-1})(1-q^{2n})(1+q^{4n-2})(1-q^{4n})}{(1-q^{2n-1})(1+q^{2n})(1-q^{4n-2})(1+q^{4n})} \\
& = & \prod_{n\geq 1}\, \frac{(1+q^n)(1-q^{2n})^2 (1-q^{4n})^3}{(1-q^n)(1-q^{8n})^2} \\
& = & \prod_{n\geq 1}\, \frac{(1-q^{2n})^3 (1-q^{4n})^3}{(1-q^n)^2(1-q^{8n})^2} \, .
\end{eqnarray*}

Introducing Ramanujan's $\psi$-function
$\psi(q) = \sum_{n \geq 0}\, q^{n(n+1)/2}$, we obtain
\begin{eqnarray*}
\varphi(q^4) + 2q \psi(q^8) 
& = & \sum_{n\in \ZZ}\, q^{(2n)^2} + 2 \sum_{n \geq 0}\, q^{4n(n+1)+1} \\
& = & \sum_{n\in \ZZ}\, q^{(2n)^2} + 2 \sum_{n \geq 0}\, q^{(2n+1)^2} = \varphi(q).
\end{eqnarray*}
Similarly, $\varphi(q^4) - 2q \psi(q^8) = \varphi(-q)$. Therefore,
$\varphi(\pm q^2) = \varphi(q^8) \pm 2q \psi(q^{16})$.
Using the previous identities, we obtain
\begin{multline*}
\varphi(q) \varphi(q^2) \\
= \varphi(q^4) \varphi(q^8) + 2q \psi(q^8) \varphi(q^8) + 2q^2 \psi(q^{16}) \varphi(q^4) + 4q^3 \psi(q^8) \varphi(q^{16})
\end{multline*}
and
\begin{multline*}
\varphi(-q) \varphi(-q^2) \\
= \varphi(q^4) \varphi(q^8) - 2q \psi(q^8) \varphi(q^8) - 2q^2 \psi(q^{16}) \varphi(q^4) + 4q^3 \psi(q^8) \varphi(q^{16}).
\end{multline*}

Now, in view of the definitions of~$\varphi$ and~$\psi$,
\begin{eqnarray*}
\varphi(q^4) \varphi(q^8) & = & \sum_{n\geq 0} \, b_n q^{4n} , \qquad
\psi(q^8) \varphi(q^8) =  \sum_{n\geq 0} \, b'_n q^{4n} , \\
\psi(q^{16}) \varphi(q^4) & = & \sum_{n\geq 0} \, b''_n q^{4n} , \qquad
\psi(q^8) \varphi(q^{16}) = \sum_{n\geq 0} \, b'''_n q^{4n} ,
\end{eqnarray*}
where $b_n, b'_n, b''_n, b'''_n$ are non-negative integers.
Therefore,
\begin{eqnarray*}
1 + \sum\, a_4(n) \, q^n & = & \varphi(-q) \varphi (-q^2) \\
& = & \sum_{n\geq 0} \, \left( b_n q^{4n} -2 b'_n q^{4n+1} -2 b''_n q^{4n+2} + 4 b'''_n q^{4n+3} \right) .
\end{eqnarray*}
We deduce 
\begin{eqnarray*}
1 + \sum\, |a_4(n)| \, q^n 
& = & \sum_{n\geq 0} \, \left( b_n q^{4n} + 2 b'_n q^{4n+1} + 2 b''_n q^{4n+2} + 4 b'''_n q^{4n+3} \right) \\
& = & \varphi(q) \varphi (q^2), \\
\end{eqnarray*}
which completes the proof.
\end{proof}

\begin{rem}
Using the computations above, we can easily express the two generating functions considered in this appendix
in terms of Dedekind's eta function $\eta(z) = q^{1/24} \prod_{n \geq 1}\, (1-q^n)$ where $q = e^{2\pi iz}$. 
Indeed,
\begin{equation*}
1 + \sum_{n\geq 1} \, a_4(n) \, q^n = \frac{\eta(z)^2 \, \eta(2z)}{\eta(4z)}
\quad\text{and}\quad
1 +  \sum_{n\geq 1} \, |a_4(n)| \, q^n = \frac{\eta(2z)^3 \, \eta(4z)^3}{\eta(z)^2 \, \eta(8z)^2} \, .
\end{equation*}
\end{rem}

\section*{Acknowledgement}

We are grateful to Giuseppe Ancona, Pierre Baumann, Olivier Benoist, Fran\c cois Bergeron, Mark Haiman, G\"unter K\"ohler,
Emma\-nuel Letellier and Luca Migliorini for helpful discussions.
We also thank Michael Somos for various comments and for the idea on which Appendix\,\ref{sec-Somos} is based.

The second-named author is grateful to the Universit\'e de Strasbourg for the invited professorship
which allowed him to spend the month of June~2014 at IRMA; 
he was also supported by NSERC (Canada).



\begin{thebibliography}{99}

\bibitem{CESM}
R. Chapman, K. Ericksson, R. P. Stanley, R. Martin,
The American Mathematical Monthly, Vol. 109, No. 1 (Jan.~2002), p. 80.

\bibitem{CS}
J. H. Conway, N. J. A. Sloane,  
\emph{Sphere packings, lattices and groups}
(with additional contributions by E.~Bannai, J.~Leech, S. P. Norton, A. M. Odlyzko, R. A. Parker, L.~Queen and B. B. Venkov),
Grundlehren der mathematischen Wissenschaften,~290, Springer-Verlag, New York,~1988.

\bibitem{De}
P. Deligne, 
\emph{La conjecture de Weil}, Inst. Hautes \'Etudes Sci. Publ. Math. 43 (1974), 273--307. 

\bibitem{Di}
L. E. Dickson, \emph{Introduction to the theory of numbers},
The University of Chicago Press, Chicago, Illinois, Sixth impression,~1946.

\bibitem{Dw}
B. Dwork, \emph{On the rationality of the zeta function of an algebraic variety},
Amer.\ J.~Math.\ 82 (1960), 631--648. 

\bibitem{Fi}
N. J. Fine, \emph{Basic hypergeometric series and applications},
Mathematical Surveys and Monographs, 27, Amer. Math. Soc., Providence, RI, 1988.

\bibitem{Gr}
A.~Grothendieck, 
\emph{Formule de Lefschetz et rationalit\'e des fonctions~$L$},
S\'eminaire Bourbaki, Vol.\,9, Exp. No.~279, 41--55, W. A. Benjamin, New York--Amsterdam, 1966.

\bibitem{HW}
G. H. Hardy, E. M. Wright, 
\emph{An introduction to the theory of numbers}, 3rd~ed., Clarendon Press, Oxford,~1954.

\bibitem{KRidealcount}
C.~Kassel, C.~Reutenauer,
\emph{Counting the ideals of given codimension of the algebra of Laurent polynomials in two variables},
arXiv:1505.07229v4 (28~October 2016).

\bibitem{KReta}
C.~Kassel, C.~Reutenauer,
\emph{The Fourier expansion of $\eta(z) \eta(z) \eta(z) /\eta(z)$}, 
arXiv:1603.06357v2 (8~July 2016).

\bibitem{Ko}
G.~K\"ohler, \emph{Eta products and theta series identities}
Springer Monographs in Mathematics. Springer, Heidelberg, 2011.

\bibitem{OEIS}
\emph{The On-Line Encyclopedia of Integer Sequences}, published electronically at http://oeis.org.

\bibitem{Ru}
J. S. Rutherford, 
\emph{Sublattice enumeration~IV. 
Equivalence classes of plane sublattices by parent Patterson symmetry and colour lattice group type}, 
Acta Crystallogr.\ Sect.~A~65 (2009), no. 2, 156--163.

\bibitem{Sl}
N. J. A. Sloane,
\emph{Theta series and magic numbers for diamond and certain ionic crystal structures},
J.~Math.\ Phys.\ 28 (1987), no.~7, 1653--1657. 

\bibitem{So}
M. Somos, \emph{A multisection of $q$-series} (11 Dec. 2014), http://somos.crg4.com/multiq.html.

\bibitem{Va}
J. E. Vatne,
\emph{The sequence of middle divisors is unbounded}, arXiv:1607.02122v1 (7~July 2016),
to appear in J.~Number Theory.

\end{thebibliography}
\end{document}